\documentclass[]{article}
\pdfoutput=1
\usepackage{amsmath}
\usepackage{amsthm}
\usepackage{amsfonts}
\usepackage{graphicx}
\usepackage{xspace}
\usepackage{xcolor}
\usepackage{hyperref}

\hypersetup{
  colorlinks=true,
  citecolor=darkgreen}

\definecolor{darkgreen}{rgb}{0,0.7,0}

\renewcommand{\_}{\hbox{-}}

\makeatletter
\def\blfootnote{\xdef\@thefnmark{}\@footnotetext}
\newcommand*{\coloneq}{\mathrel{\rlap{%
  \raisebox{0.3ex}{$\m@th\cdot$}}%
  \raisebox{-0.3ex}{$\m@th\cdot$}}%
  =}
\makeatother

\theoremstyle{plain}
\newtheorem{lemma}[equation]{Lemma}
\newtheorem{prop}[equation]{Proposition}
\newtheorem{cor}[equation]{Corollary}

\mathchardef\ldotp="613A   
\newcommand{\dd}{\mathbin{\ldotp\!\ldotp}}

\newcommand{\reals}{\mathbb{R}}
\newcommand{\ints}{\mathbb{Z}}

\newcommand{\s}{\kern -0.5pt 's\xspace}
\renewcommand{\t}{\kern -0.5pt 't\xspace}

\newcommand{\hsmash}[1]{\hbox to 0pt{\hss #1 \hss}}

\newcommand{\Conv}{\mathcal{A}}
\newcommand{\ConvEuc}{\mathcal{A}_{\mkern -1mu E\mkern -2mu}}
\newcommand{\ConvHyp}{\mathcal{A}_{\mkern -2mu H\mkern -2mu}}
\newcommand{\PolygonSpace}{\mathcal{L}}

\newcommand{\nlow}{n}
\newcommand{\nthin}{{n\mkern -1mu}}

\newcommand{\edge}{\varepsilon}

\newcommand{\sbargon}{$\bar{\mathbf{s}}$\_gon\xspace}

\newcommand{\shatgon}{$\hat{\mathbf{s}}$\_gon\xspace}
\newcommand{\shatgons}{$\hat{\mathbf{s}}$\_gons\xspace}

\begin{document}

\title{Dual-cyclic polytopes of convex planar\\
polygons with fixed vertex angles}
\author{Lyle Ramshaw\footnote{%
\kern 0.9pt\protect\includegraphics[scale=0.06]{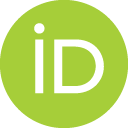}
\protect\href{https://orcid.org/0000-0002-9113-1473}{//orcid.org/0000-0002-9113-1473}}
\ and James B.~Saxe}

\maketitle

\begin{abstract}
If we fix the angles at the vertices of a convex $n$\_gon in the plane, the lengths of its $n$ edges must satisfy two linear constraints in order for it to close up.  If we also require unit perimeter, we get vectors of $n$ edge lengths that form a convex polytope of dimension $n-3$, each facet of which consists of those $n$\_gons in which the length of a particular edge has fallen to zero.  Bavard and Ghys suggest requiring unit area instead, which gives them a hyperbolic polytope.  Those two polytopes are combinatorially equivalent, so either is fine for our purposes.

Such a fixed-angles polytope is combinatorially richer when the angles are well balanced.  We say that fixed external angles are \emph{majority dominant} when every consecutive string of more than half of them sums to more than $\pi$.  When~$n$ is odd, we show that the fixed-angles polytope for any majority-dominant angles is dual to the cyclic polytope $C_{n-3}(n)$; so it has as many vertices as are possible for any polytope of dimension $n-3$ with~$n$ facets.  To extend that result to even~$n$, we require that the angles also have \emph{dipole tie-breaking}: None of the $n$ strings of length $n/2$ sums to precisely $\pi$, and the $n/2$ that sum to more than $\pi$ overlap as much as possible, all containing a particular angle.

Fixing the vertex angles is uncommon, however; people more often fix the edge lengths.  That is harder, in part because fixed-lengths $n$\_gons may not be convex, but mostly because fixing the lengths constrains the angles nonlinearly --- so the resulting moduli spaces, called \emph{polygon spaces}, are curved.  Consider the polygon space for some edge lengths that sum to~$2\pi$.  Using Schwarz--Christoffel maps, Kapovich and Millson show that the subset of that polygon space in which the $n$\_gons are convex and traversed counterclockwise is homeomorphic to the fixed-angles polytope above, for those same fixed values.  Each such subset is thus a topological polytope; and it is dual cyclic whenever the fixed lengths are majority dominant and, for even~$n$, have dipole tie-breaking.
\end{abstract}

\blfootnote{MSC-class: 52B11 (Primary) 51M20, 52B05 (Secondary)}

\section{Introduction}
\label{sect:Intro}

\subsection{The fixed-angles polytopes}

The external angles at the vertices of a convex planar $n$\_gon must sum to $2\pi$.  For convenience, we guarantee that sum via rescaling:  Given $n$ positive numbers~$s_1$ through~$s_n$ with sum $S\coloneq s_1+\cdots+s_n$, we consider those $n$\_gons in the Euclidean plane whose $k^\text{th}$ vertices have external angle~$2\pi s_k/S$, for all $k$.  We view the~$s_k$ cyclically, with $s_n$ followed by $s_1$; so we group them into a necklace, written in square brackets: $\mathbf{s}=[s_1,\ldots,s_n]=[s_2,\ldots,s_n,s_1]$.

In order for such a fixed-angles $n$\_gon to close up in $x$ and in~$y$, its edge lengths must satisfy two linear constraints; let $L_\mathbf{s}$ be the subspace of $\reals^n$ where those constraints hold, which has dimension $n-2$.  Since edge lengths must be nonnegative, the vectors of edge lengths form the cone $K_\mathbf{s}\coloneq L_\mathbf{s}\cap \reals_{\ge 0}^n$.

The $n$\_gons along any ray in the cone $K_\mathbf{s}$ differ only by magnification, and we can select one representative from each ray by requiring either unit perimeter or unit area, as we discuss in Section~\ref{sect:EuclideanPolytopes}.  If we require unit perimeter, our vectors of edge lengths form a Euclidean convex polytope of dimension $n-3$.  We denote it $\ConvEuc^n[\mathbf{s}]$, the letter ``$\mathcal{A}$'' reminding us that the vertex Angles are here held fixed.  Bavard and Ghys~\cite{BavardGhys} suggest requiring unit area instead, which gives them a hyperbolic polytope $\ConvHyp^n[\mathbf{s}]$.  Since the polytopes $\ConvEuc^n[\mathbf{s}]$ and $\ConvHyp^n[\mathbf{s}]$ are combinatorially equivalent, we can, for our combinatorial purposes, denote the \emph{fixed-angles polytope} for the angles in $\mathbf{s}$ simply as $\Conv^\nthin[\mathbf{s}]$, meaning either one.  The Euclidean is more elementary, but the hyperbolic is more elegant.

The total turn from one edge of a fixed-angles $n$\_gon to another is controlled by the sum of a string of consecutive entries of~$\mathbf{s}$; and it is key how those sums compare to $S/2$.  When such a sum equals $S/2$, the $n$\_gon has a pair of antiparallel edges --- a case that requires special treatment.  The fixed-angles polytope has a vertex for each way of partitioning $\mathbf{s}$ into three strings whose sums are less than~$S/2$.  The three edges that constitute the gaps between those strings can extend to form a unit-perimeter (or unit-area) triangle in which the other $n-3$ edges have shrunk to points.  The resulting vertex of the fixed-angles polytope is simple, since it lies on $n-3$ facets in dimension $n-3$.  And the number of such simple vertices turns out to be maximized when the angles in $\mathbf{s}$ are well balanced in an appropriate sense.

\subsection{Cyclic polytopes}
Recall~\cite{Ziegler} that a \emph{cyclic polytope} $C_d(m)$, for $1\le d<m$, is the convex hull of $m$ points along the \emph{moment curve} in $\mathbb{R}^d$, the curve given parametrically by $t\mapsto (t,t^2,\ldots,t^d)$.  The combinatorial structure of that convex hull doesn\t depend upon which $m$ points we choose, as long as they are distinct.  The cyclic polytope $C_d(m)$ is \emph{neighborly}, meaning that every set of at most $d/2$ vertices forms a face.  It also has, simultaneously for all~$k$, as many $k$\_faces as are possible for any $d\mkern -1mu$\_polytope with $m$ vertices.  That property has no standard name, but let\s call it being \emph{max-faced}.

The cyclic polytope $C_d(m)$ is particularly straightforward when $m-d$ is tiny.  The polytope $C_d(d+1)$, for example, is just the $d\mkern -1mu$\_simplex $\Delta_d$, which is self-dual.  A standard exercise~\cite[p.~24]{Ziegler} shows that the polytope $C_d(d+2)$ is $(\Delta_{\lfloor d/2\rfloor}\times \Delta_{\lceil d/2\rceil})^\Delta$, the dual of the Cartesian product of two simplices of nearly equal dimension.  We shed some light on the next-more-complicated case by showing that
$\ConvEuc^{d+3}[\mathbf{s}]^\Delta$, the dual of a fixed-lengths polytope, is an instance of $C_d(d+3)$ for any $d+3$ fixed angles $\mathbf{s}$ that are appropriately well balanced.

While the polytopes $C_d(d+1)$ and $C_d(d+2)$ have lots of symmetries, the level of symmetry that remains in $C_d(m)$ when $m\ge d+3$ depends on the parity of the dimension~$d$~\cite{Kaibel}.  When~$d$ is even, the polytope $C_d(m)$ has $2m$ combinatorial automorphisms that dihedrally permute its $m$ vertices.  When~$d$ is odd, on the other hand, it has just four:  We can swap the first and last vertices, reverse the order of all of the vertices, or both, or neither.  For $d$ of either parity, there are instances of $C_d(m)$ whose combinatorial automorphisms all come from Euclidean symmetries.

\subsection{When are fixed angles well balanced?}

In Section~\ref{sect:BreakingTies}, we study how the combinatorial structure of the fixed-angles polytope $\Conv^\nthin[\mathbf{s}]$ changes as the angles in $\mathbf{s}$ vary.  The structure changes each time that the sum of some substring of $\mathbf{s}$ passes through a tie with the sum of the complementary substring, that tie generating a pair of antiparallel edges.  Of the structures before and after that change, the richer is the one in which the longer substring has the larger sum, so that $\mathbf{s}$ is better balanced.

We say that the angles in~$\mathbf{s}$ are \emph{majority dominant} when every consecutive substring of length more than~$n/2$ sums to more than $S/2$.  The complementary strings, which are minorities, then sum to less than $S/2$.  Majority dominance says nothing about the strings of length precisely $n/2$, which exist only when~$n$ is even.  When $n$ is odd and~$\mathbf{s}$ is majority dominant, we show that the fixed-angles polytope is dual cyclic by constructing, in Section~\ref{sect:nOdd}, an explicit duality between the face lattices of $\Conv^\nthin[\mathbf{s}]$ and $C_{n-3}(n)$. 

When $\mathbf{s}$ is majority dominant but $n$ is even, the dual polytope $\Conv^\nthin[\mathbf{s}]^\Delta$ may not be cyclic, but it is neighborly.  If there are no substrings of $\mathbf{s}$ that sum to~$S/2$, it is also simplicial, which implies that it is max-faced, by the McMullen upper-bound theorem~\cite[p.~254]{Ziegler}.  We show in Section~\ref{sect:nEven} that $\Conv^\nthin[\mathbf{s}]^\Delta$, beyond being max-faced, is precisely an instance of $C_{n-3}(n)$ if we require that $\mathbf{s}$ also have \emph{dipole tie-breaking}, meaning that, of its~$n$ substrings of length $n/2$, the $n/2$ that contain some particular entry all sum to more than~$S/2$.  The others then contain the opposite entry and sum to less than $S/2$.

What can we conclude about equiangular $n$\_gons, whose fixed angles are perfectly balanced?  Since the sequence $(1,\ldots,1)$ is majority dominant, the fixed-angles polytope $\Conv^\nthin[1,\ldots,1]$ for equiangular $n$\_gons is dual cyclic when $n$ is odd, with $\Conv^\nthin[1,\ldots,1]^\Delta$ being an instance of $C_{n-3}(n)$.\footnote{The authors thank G\"unter Ziegler for pointing out to them, in a 2017 email, that the fixed-angles polytope $\Conv^7[1,\ldots,1]$ for equiangular heptagons is dual to $C_4(7)$.}  When $n$ is even, the polytope $\Conv^\nthin[1,\ldots,1]^\Delta$ continues to have symmetries that permute its $n$ vertices dihedrally; it thus can\t be dual cyclic, since $C_{n-3}(n)$ then has only four combinatorial automorphisms.  Instead, the sequence $(1,\ldots,1)$ acquires substrings that sum to $S/2$ when $n$ is even; so the polytope $\Conv^\nthin[1,\ldots,1]^\Delta$, while still neighborly, is not simplicial and is hence neither max-faced nor cyclic.

\subsection{The fixed-lengths-convex top-polytopes}

Fixing the edge lengths of polygons is more popular than fixing their vertex angles.  Fixed edge lengths arise, for example, when analyzing a cyclic planar linkage with rigid links and revolute joints, one of the many kinds of linkages studied in robotics~\cite{Farber}.  Fixing the edge lengths leads to harder problems than fixing the vertex angles, both because the resulting $n$\_gons may not be convex and, more importantly, because the constraints that fixed lengths impose on the angles are nonlinear, so the resulting moduli spaces are curved.

For now, we impose convexity as a side constraint.  Call an $n$\_gon \emph{ccw\_convex} when it is convex and traversed counterclockwise, so its external angles are nonnegative.  Given a necklace $\mathbf{s}\coloneq [s_1,\ldots,s_n]$ of positive numbers, we consider those ccw-convex planar $n$\_gons whose $k^\text{th}$ edge has length $s_k$, for all $k$; and we say that two $n$\_gons have the same \emph{shape} when they differ by translation or rotation.  The resulting moduli space of shapes is a topological polytope, as we discuss next; so we call it the \emph{fixed-lengths-convex top-polytope} for $\mathbf{s}$.  We denote it $\PolygonSpace^n_c[\mathbf{s}]$, where the ``$\mathcal{L}$'' reminds us that the edge Lengths are here held fixed, and the subscript $c$ means ccw-convex.

Kapovich and Millson~\cite{KapMill95} use Schwarz--Christoffel maps to transform each ccw\_convex $n$\_gon whose edge lengths are given by~$\mathbf{s}$ into an $n$\_gon whose external angles are proportional to~$\mathbf{s}$ (this $n$\_gon also being ccw-convex).  As we review in Section~\ref{sect:FixedLengths}, they thereby construct a homeomorphism $h_{\mathbf{s}}\colon \PolygonSpace_c^n[\mathbf{s}]\to\Conv^\nthin[\mathbf{s}]$ from the fixed-lengths-convex top-polytope $\PolygonSpace_c^n[\mathbf{s}]$ to the fixed-angles polytope~$\Conv^\nthin[\mathbf{s}]$.  The space $\PolygonSpace_c^n[\mathbf{s}]$ is thus homeomorphic to a polytope by a preferred homeomorphism, which makes it a \emph{topological polytope}.  By our earlier results, it is dual to the cyclic polytope $C_{n-3}(n)$ whenever the fixed lengths $\mathbf{s}$ are majority dominant and, if $n$ is even, also have dipole tie-breaking.  

When $\mathbf{s}$ has no substrings with the same sums as their complements, both the  fixed-lengths-convex top-polytope $\PolygonSpace_c^n[\mathbf{s}]$ and the fixed-angles polytope $\Conv^\nthin[\mathbf{s}]$ are smooth manifolds with corners.  And it seems likely that they are actually diffeomorphic, as smooth manifolds with corners.   But we close our final Section~\ref{sect:FixedLengths} by warning that the Kapovich--Millson homeomorphism $h_{\mathbf{s}}\colon \PolygonSpace_c^n[\mathbf{s}]\to\Conv^\nthin[\mathbf{s}]$ is not a diffeomorphism, even in those well-behaved cases.

\subsection{Gluing polytopes together}

In this paper, we study our polytopes and our top-polytopes one at a time.  But they are often glued together into larger spaces.

In the fixed-angles polytope $\Conv^\nthin[\ldots,p,q,\ldots]$, consider the facet where the edge joining the vertices with angles $p$ and $q$ has shrunk to a point.  Those vertices then coalesce into a \emph{supervertex} with angle $p+q$; so that facet is a copy of the polytope $\Conv^{n-1}[\ldots,p+q,\ldots]$.  The polytope $\Conv^\nthin[\ldots,q,p,\ldots]$, in which~$p$ and~$q$ have swapped places, also has a copy of $\Conv^{n-1}[\ldots,p+q,\ldots]$ as a facet; and we can glue $\Conv^\nthin[\ldots,p,q,\ldots]$ to $\Conv^\nthin[\ldots,q,p,\ldots]$ along those matching facets.

Given $n$ positive numbers that sum to $2\pi$, suppose that we assemble them into necklaces in all $(n-1)!$ cyclic orders, construct the resulting fixed-angles polytopes, and then glue them together along all of their matching pairs of facets.  The result is a Euclidean (or hyperbolic) cone-manifold that we can view as the moduli space of unit-perimeter (or unit-area) planar $n$\_gons whose vertices have our $n$ specified external angles, in any order.  The researchers who do such gluing use the hyperbolic polytopes, since their more elegant geometries lead to cone-manifolds that may actually be orbifolds or manifolds~\cite{Fillastre}.\footnote{Also, they often want to identify each fixed-angles $n$\_gon with its reflection; so they assemble their $n$ fixed angles, not into into $(n-1)!$ necklaces, but into $(n-1)!/2$ bracelets.}

The fixed-lengths world typically develops in the opposite order.  Rather than constructing the fixed-lengths-convex top-polytopes and gluing them together, people instead construct the big space at the outset.  The \emph{polygon space}~\cite{FarberSchutz,HausmannRodriguez,KapMill95} for the \emph{length vector}~$\mathbf{s}$ is the moduli space of shapes of planar $n$\_gons whose $k^\text{th}$ edge has length $s_k$, for all $k$, with no requirement that they be ccw-convex.  We denote that space simply as $\PolygonSpace^\nthin[\mathbf{s}]$, with no subscript of ``$c$''\kern -1pt.  The polygon space $\PolygonSpace^n[\mathbf{s}]$ is a compact, orientable smooth manifold of dimension $n-3$, with the possible exception of some isolated quadratic singular points.  Only as an after-thought do people partition this manifold, based on the cyclic order of the complex phases of the $n$\_gon\s edges, into $(n-1)!$ top-polytopes.

\section{The fixed-angles polytopes}
\label{sect:EuclideanPolytopes}

We now start our work in earnest by fixing the vertex angles of polygons in the Euclidean plane.  Given a necklace $\mathbf{s}\coloneq[s_1,\ldots,s_n]$ of $n$ positive numbers with sum $S\coloneq s_1+\cdots+s_n$, we define an \emph{\shatgon} to be an $n$\_gon whose $k^\text{th}$ vertex has external angle $2\pi s_k/S$.  The hat accent indicates that it is the vertex angles that are fixed.  When we later use the $s_k$ as the fixed lengths of the edges of an $n$\_gon, we will call the result an \emph{\sbargon}, with a bar accent.

While the angle at the $k^\text{th}$ vertex of an \shatgon is the quotient $2\pi s_k/S$, we refer to $s_k$ itself as the \emph{shangle} of that vertex, where ``shangle'' abbreviates ``share of the angle''\kern -1pt.  So the external angles are proportional to the shangles.

Each edge of an \shatgon constitutes a gap between elements of the shangle necklace $\mathbf{s}$.  We denote by $\edge_k$ the edge that connects the vertices with shangles $s_{k-1}$ and $s_k$.\footnote{The choice of how to offset the numbering of the edges from that of the vertices seems unavoidably arbitrary.  When $n$ is odd, it works well to let $k^\text{th}$ edge be the one opposite the $k^\text{th}$ vertex.  Indeed, this is quite standard for triangles, with the edges $a$, $b$, and $c$ being opposite the vertices $A$, $B$, and $C$.  But $n$ is not always odd.} We describe the shape of an \shatgon as the vector $\pmb{\ell}= (\ell_1,\ldots,\ell_n)$, where $\ell_k$ is the length of the edge $\edge_k$.  Closure in $x$ and in $y$ imposes two linear constraints, limiting the vector~$\pmb{\ell}$ to a linear subspace $L_{\mathbf{s}}\subset\mathbb{R}^n$ of dimension $n-2$.  The further constraint that $\ell_k\ge 0$ for all $k$ limits the shapes of \shatgons to the cone $K_{\mathbf{s}}\coloneq L_{\mathbf{s}}\cap \reals^n_{\ge 0}$.  

We next select one representative from each ray in $K_{\mathbf{s}}$ of similar \shatgons.  One option selects, from each ray, the \shatgon that has unit perimeter.  The resulting length vectors form a Euclidean convex polytope that we denote $\ConvEuc^n[\mathbf{s}]$, lying in the unit-perimeter $(n-3)$\_flat given by $\sum_k\ell_k=1$.  As a more sophisticated alternative, Bavard and Ghys~\cite{BavardGhys} select, as their representative, the \shatgon with unit area.  The functional on the $(n-2)$\_space $L_\mathbf{s}$ that measures area is a quadratic form with signature $(1,n-3)$; so the unit-area subset of $L_\mathbf{s}$ forms a hyperboloid of two sheets, each sheet of which provides a Minkowski model for real hyperbolic $(n-3)$\_space.  The nonnegative orthant $\reals_{\ge 0}^n$ cuts out, from one of those sheets, a hyperbolic convex polytope $\ConvHyp^n[\mathbf{s}]$.  We call the polytopes $\ConvEuc^n[\mathbf{s}]$ and $\ConvHyp^n[\mathbf{s}]$ the \emph{fixed-angles polytopes} for the shangle necklace $\mathbf{s}$.

As an example, Figure~\ref{fig:ConvComb} shows the Euclidean fixed-angles polytope $\ConvEuc^5[\mathbf{q}]$ whose shangle necklace is $\mathbf{q}=[1,3,2,4,5]$.  That polytope is a polygon because $n-3=2$; and it is a quadrilateral, rather than a pentagon, because the edge in a $\hat{\mathbf{q}}$\_gon between the vertices with shangles $4$ and $5$ (and hence with angles $96^\circ$ and $120^\circ$) --- the brown edges drawn as the bases in Figure~\ref{fig:ConvComb} --- can\t shrink to a point while maintaining unit perimeter.

\begin{figure}[t]
\small
    \begin{center}
		\includegraphics[scale=0.9]{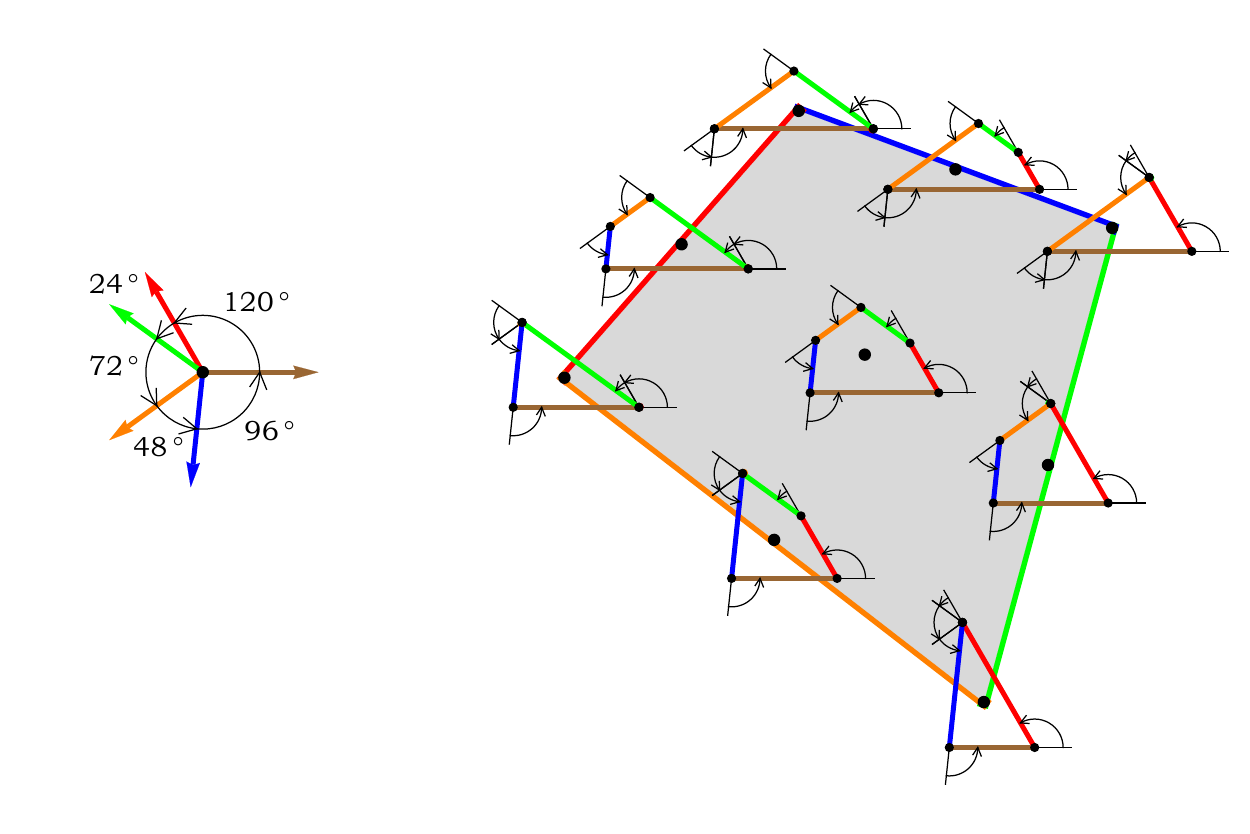}
	\end{center}
    \vspace*{-10 pt}
	\caption{The fixed-angles polytope $\ConvEuc^5[\mathbf{q}]$ for $\mathbf{q}=[1,3,2,4,5]$ is a quadrilateral.  Rescaling $\mathbf{q}$ to sum to $360$ degrees gives $[24,72,48,96,120]$, by the way.  The sample $\hat{\mathbf{q}}$\_gons shown, which all have unit perimeter, are at the four vertices of the quadrilateral $\ConvEuc^5[\mathbf{q}]$, the midpoints of its four edges, and the centroid of its four vertices (which differs from the centroid of its interior).}
	\label{fig:ConvComb}
\end{figure}

As we move around in the quadrilateral $\ConvEuc^5[\mathbf{q}]$, the lengths of the $\hat{\mathbf{q}}$\_gon edges vary affinely.  The length of a $\hat{\mathbf{q}}$\_gon\s red edge varies linearly with its distance away from the red side of $\ConvEuc^5[\mathbf{q}]$, the side along which the red edges have shrunk to points; and similarly for green, orange, and blue.  As for brown, the quadrilateral $\ConvEuc^5[\mathbf{q}]$ is rotated in Figure~\ref{fig:ConvComb} so that the length of a $\hat{\mathbf{q}}$\_gon\s brown base varies affinely with its height on the page; but the horizontal line where that length falls to $0$ passes well below the entire quadrilateral $\ConvEuc^5[\mathbf{q}]$.

\subsection{Trigon and digon vertices}

The structure of the fixed-angles polytope $\ConvEuc^n[\mathbf{s}]$ depends upon the sums of consecutive substrings of the shangle necklace $\mathbf{s}$.  We refer to the sum of such a substring as its \emph{weight}, as distinguished from its length.  Given a necklace $\mathbf{s}=[s_1,\ldots,s_n]$ with total weight $s_1+\cdots+s_n=S$, we call a substring of~$\mathbf{s}$ either \emph{light}, \emph{tied}, or \emph{heavy} according as its weight is less than, equal to, or greater than $S/2$.  If any substring of $\mathbf{s}$ is tied, then its complement is also tied, and they form a \emph{substring tie} in $\mathbf{s}$.  The \emph{width} of that tie, which we denote $w$, is the minimum of the lengths of the two complementary substrings; so $w\le n/2$.

When an edge $\edge_k$ of the \shatgon shrinks to a point, the vertices at its two ends, which have shangles $s_{k-1}$ and $s_k$, coalesce into a \emph{supervertex} of combined shangle $s_{k-1}+s_k$.  This can\t happen when $s_{k-1}+s_k> S/2$, that is, when the string $(s_{k-1}, s_k)$ is heavy, since external angles of convex polygons can\t exceed~$\pi$.  Restricting the edge $\edge_k$ to have length $0$ constrains us to the subset $\ConvEuc^{n-1}[\mathbf{s}']\subseteq\ConvEuc^n[\mathbf{s}]$, where $\mathbf{s}'$ denotes $\mathbf{s}$ with its separate shangles $s_{k-1}$ and $s_k$ replaced by their sum.  That subset is empty when $s_{k-1}+s_k>S/2$, is the single point that constitutes all of $\ConvEuc^n[\mathbf{s}]$ when $s_{k-1}+s_k=S/2$, and is otherwise a facet of $\ConvEuc^n[\mathbf{s}]$.  More generally, the edges $\edge_i$ through $\edge_j$, for $i\le j$, can shrink to points simultaneously only when the string $(s_{i-1},\ldots,s_j)$ isn\t heavy.

A typical vertex of the fixed-angles polytope $\ConvEuc^n[\mathbf{s}]$ arises from some way to partition the shangle necklace $\mathbf{s}$ into three light substrings.  The three edges in the gaps between those substrings can extend to form a triangular \shatgon of unit perimeter, with the other $n-3$ edges all of length $0$.  The corresponding point in the polytope $\ConvEuc^n[\mathbf{s}]$ is thus a simple vertex, which we call a \emph{trigon vertex}.

Each substring tie in the shangle necklace $\mathbf{s}$ generates an atypical vertex.  The two edges in the gaps between the tied substrings are antiparallel, since the total turn involved in going from each to the other is $\pi$.  Those two edges can thus grow to length $\frac{1}{2}$ while all of the other edges shrink to length $0$, producing an \shatgon that is a digon of unit perimeter.  We call the resulting vertex of the fixed-angles polytope $\ConvEuc^n[\mathbf{s}]$ a \emph{digon vertex}.

If a necklace $\mathbf{s}$ is free of substring ties, then all of the vertices of $\ConvEuc^n[\mathbf{s}]$ are trigon vertices, so $\ConvEuc^n[\mathbf{s}]$ is a simple polytope.   A digon vertex may or may not be simple.  To explore that question, we investigate the vertex figures of the vertices of $\ConvEuc^n[\mathbf{s}]$.

\subsection{The vertex figures}

The vertex figure of any simple vertex is a simplex, and we can see this concretely for a trigon vertex of $\ConvEuc^n[\mathbf{s}]$.  We cut off a trigon vertex by constraining the lengths of its three long edges to sum to $1-\epsilon$, for some small $\epsilon$.  The other $n-3$ edges must then have lengths that sum to $\epsilon$, and they can partition that~$\epsilon$ among themselves however they like, closure then being achieved by suitable adjustments of the lengths of the long edges.  So the vertex figure is a $\Delta_{n-4}$.

Now consider a digon vertex that results from a substring tie of width $w$.  Suppose that we constrain the perpendicular distance between the two antiparallel edges to be some small~$\epsilon$.  The remaining $n-2$ edges form two polygonal chains that connect the ends of the antiparallel edges, say the left chain with $w-1$ edges and the right with $n-w-1$.  The edges in each chain must travel an overall perpendicular distance of $\epsilon$, and they can partition that travel among themselves however they like.  The structure of the left chain thus corresponds to a point in the simplex $\Delta_{w-2}$ and that of the right chain to a point in $\Delta_{n-w-2}$.  So the overall vertex figure is the product $\Delta_{w-2}\times\Delta_{n-w-2}$.  

The product $\Delta_{w-2}\times\Delta_{n-w-2}$ is empty when $w=1$, since $\Delta_{-1}=\emptyset$; and it is the simplex $\Delta_{n-4}$ when $w=2$, since $\Delta_{0}$ is a single point.  Once $w\ge 3$ (which implies that $n-w\ge 3$ also), that vertex figure is a polytope of dimension $n-4$ with $n-2$ facets.  So a digon vertex of $\ConvEuc^n[\mathbf{s}]$ arising from a tie of width $w\ge 3$ is nonsimple.

\subsection{Classes of fixed-angles polytopes}  
What classes arise for fixed-angles polytopes?  The polytope $\ConvEuc^n[\mathbf{s}]$ is empty when $\max(s_1,\ldots,s_n)>S/2$, because no convex $n$\_gon can have an external angle that exceeds $\pi$.  It is a single point when $\max(s_1,\ldots,s_n)=S/2$.  The shangle necklace $\mathbf{s}$ then has a substring tie of width $w=1$, and the polytope $\ConvEuc^n[\mathbf{s}]$ is the isolated digon vertex arising from that tie.  When $\max(s_1,\ldots,s_n)<S/2$, all of the external angles are less than $\pi$.  Given any circle, we can construct an \shatgon whose edges are all tangent to that circle; such a polygon, by the way, is called \emph{tangential}.  And we can rescale our tangential \shatgon to have unit perimeter.  Since all of its edges have positive length, a small enough ball around that \shatgon in the unit-perimeter $(n-3)$\_flat $\sum_k\ell_k=1$ will lie in the interior of $\ConvEuc^n[\mathbf{s}]$, which is thus a polytope of the full dimension $n-3$.

When $\ConvEuc^n[\mathbf{s}]$ has dimension $n-3$, all of its trigon vertices are simple, and any digon vertices arising from ties of width $2$ are also simple.  For example, replacing the necklace $\mathbf{q}=[1,3,2,4,5]$ in Figure~\ref{fig:ConvComb} with $\mathbf{q}'=[2,4,3,4,5]$ would give a tie of width $2$ because $2+4+3=4+5$.  The horizontal line where the brown base falls to length zero would then have risen to just touch the lowest, orange-to-green vertex of the fixed-angles quadrilateral $\ConvEuc^5[\mathbf{q}']$; but that vertex would remain simple. 

A tie of width $w\ge 3$, however, gives a digon vertex that is nonsimple, lying on $n-2$ facets rather than only $n-3$.  The first examples of this arise when $n=6$ and $w=3$.  Consider, for example, the fixed-angles polyhedron $\ConvEuc^6[1,\ldots,1]$ for equiangular hexagons.  It is combinatorially a bipyramid over a triangle --- geometrically, an equilateral triangle with cube-corners on both sides.  Each pair of antiparallel hexagon edges generates a vertex of the triangle, and four faces of $\ConvEuc^6[1,\ldots,1]$ meet at that vertex as at a vertex of an octahedron.

In summary:

\begin{prop}
Let $\mathbf{s}=[s_1,\ldots,s_n]$ be a necklace of positive numbers and let $S=s_1+\cdots+s_n$ be their sum.  The Euclidean fixed-angles polytope $\ConvEuc^n[\mathbf{s}]$
\begin{itemize}
\item is empty when $\max(s_1,\ldots,s_n)>S/2$;
\item is an isolated point when $\max(s_1,\ldots,s_n)=S/2$, that point being the digon vertex arising from the resulting substring tie of width $w=1$;
\item and has the full dimension of $n-3$ when $\max(s_1,\ldots,s_n)<S/2$.
\end{itemize}
In the third case, its trigon vertices are all simple, as are any digon vertices arising from ties of width $2$.  But a digon vertex, if any, arising from a tie of width $w\ge 3$ is nonsimple, having $\Delta_{w-2}\times \Delta_{n-w-2}$ as its vertex figure. 
\end{prop}

\subsection{The hyperbolic alternative}
\label{sect:HyperbolicPolytopes}

We focus on the Euclidean fixed-angles polytope $\ConvEuc^n[\mathbf{s}]$ in this paper, rather than on the Bavard--Ghys~\cite{BavardGhys} hyperbolic $\ConvHyp^n[\mathbf{s}]$, since the Euclidean polytopes are more elementary.  The detailed geometry of $\ConvHyp^n[\mathbf{s}]$ is prettier than that of $\ConvEuc^n[\mathbf{s}]$, however.  Indeed, each polytope $\ConvHyp^n[\mathbf{s}]$ is a \emph{$k$\_truncated $(n-3)$\_orthoscheme}, for some $k$ with $0\le k\le 2$~\cite{Fillastre}.\footnote{Truncating a vertex of a polytope typically replaces that vertex with a new facet, with a new normal vector.  In a $k$\_truncated orthoscheme, however, $k$ of the hyperplanes that would otherwise define facets are so far away from the polytope that their facets don\t arise.}  Also, there is an elegant formula for the dihedral angle between two facet hyperplanes of $\ConvHyp^n[\mathbf{s}]$ that are not orthogonal, based on the cross ratio of the slopes of four of the $\hat{\mathbf{s}}$\_gon\s edges (proven by Bavard and Ghys~\cite{BavardGhys} and nicely explained by Fillastre~\cite{Fillastre}).

A warning about digon vertices in the hyperbolic case:  We can get a digon with unit area only as a limit, with the two antiparallel edges growing to be infinitely long while all of the other edges shrink to points.  A digon vertex of $\ConvHyp^n[\mathbf{s}]$ is thus an ideal point, located out on the Cayley absolute.  That vertex is either isolated, simple, or nonsimple according as $w=1$, $w=2$, or $w\ge 3$, just as in the Euclidean case; but all digon vertices in the hyperbolic case are ideal.

As one example of the geometric elegance of the hyperbolic option, consider the fixed-angles polytope $\Conv^5[\mathbf{q}]$.  The Euclidean version $\ConvEuc^5[\mathbf{q}]$ in Figure~\ref{fig:ConvComb} is nothing special; but its hyperbolic analog $\ConvHyp^5[\mathbf{q}]$ is a Lambert quadrilateral, with three right angles (aka a \emph{singly truncated $2$\_orthoscheme}~\cite{Fillastre}).  The one angle that isn\t right is at the lowest, orange-to-green vertex, the vertex that would have been cut off by a brown edge, were there such an edge.  

As a second example, we mentioned that $\ConvEuc^6[1,\ldots,1]$, the Euclidean fixed-angles polyhedron for equiangular hexagons, consists of cube-corners on both sides of an equilateral triangle.  In the hyperbolic $\ConvHyp^6[1,\ldots,1]$, an untruncated $3$\_orthoscheme, that triangle is triply ideal; so each of the six triangular faces of $\ConvHyp^6[1,\ldots,1]$ is orthogonal to the three that it meets along an edge and is tangent at infinity to the other two, each of which~it touches just at a vertex.\footnote{Thurston~\cite[p. 515]{Thurston} discusses this hyperbolic polyhedron in a related hexagonal context.}

\section{Varying the shangles}
\label{sect:BreakingTies}

Consider a shangle necklace $\mathbf{s}$ that has a substring tie, say of width~$w$.  We can perturb $\mathbf{s}$ to eliminate that tie in either of two directions.  How do those perturbations affect the structure of the fixed-angles polytope $\ConvEuc^n[\mathbf{s}]$?

A substring tie of width $w$ in $\mathbf{s}$ gives each \shatgon a pair of antiparallel edges that are connected, say, at their left end, by a chain of $w-1$ other edges and, at their right end, by a chain of $n-w-1$.  If all of the edges in either chain shrink to points, then all of the edges in both chains must so shrink, the two paired edges must have length $\frac{1}{2}$, and we are at the associated digon vertex.  

Suppose that we now perturb $\mathbf{s}$ by slightly increasing the shangles in the tied substring of length $w$.  The total turn on the left from one of the paired edges to the other will then exceed $\pi$, so it will no longer be possible for all of the edges in the left chain to shrink to points simultaneously.  That will still be possible, though, for the $n-w-1$ edges in the right chain.  When those edges do all shrink to points simultaneously, we find ourselves on some face, say $F\mkern -2mu$, of the perturbed polytope.  The face $F$ has dimension $(n-3)-(n-w-1)=w-2$; but what is its combinatorial structure?  

When all of the edges of the right chain have length $0$, we can imagine the two paired edges as directly joined, on the right, by a supervertex whose external angle is a skosh less than $\pi$.  For perturbations that are small enough, however, adding to that external angle the fixed external angle at the other end of either of the paired edges will give a sum that exceeds $\pi$; so neither of the paired edges can shrink to points.  It follows that the face $F\mkern -2mu$, viewed in its own right as the fixed-angles polytope of this simplified $(w+1)$\_gon, has only $w-1$ facets.  So the face $F$ is combinatorially a $\Delta_{w-2}$.  

In a similar way, if we perturb $\mathbf{s}$ by slightly increasing the shangles in the substring of length $n-w$, the edges in the right chain can no longer shrink to points simultaneously; and having those in the left chain shrink simultaneously lands us on a face of the perturbed polytope that is combinatorially a $\Delta_{n-w-2}$.  Thus, perturbing $\mathbf{s}$ so as to break the tie causes the digon vertex to expand into either a $\Delta_{w-2}$ or a $\Delta_{n-w-2}$.  

If $w=n/2$, the two directions in which we can perturb are symmetric.  If $w<n/2$, the tie already implies a certain imbalance in $\mathbf{s}$.  Perturbing so as to exacerbate that imbalance expands the digon vertex into a $\Delta_{w-2}$, while perturbing to ameliorate gives us a $\Delta_{n-w-2}$.  Shangle necklaces that are better balanced thus generate fixed-angles polytopes that are numerically more complicated.

\begin{figure}\small
    \begin{center}
		\includegraphics[scale=0.95]{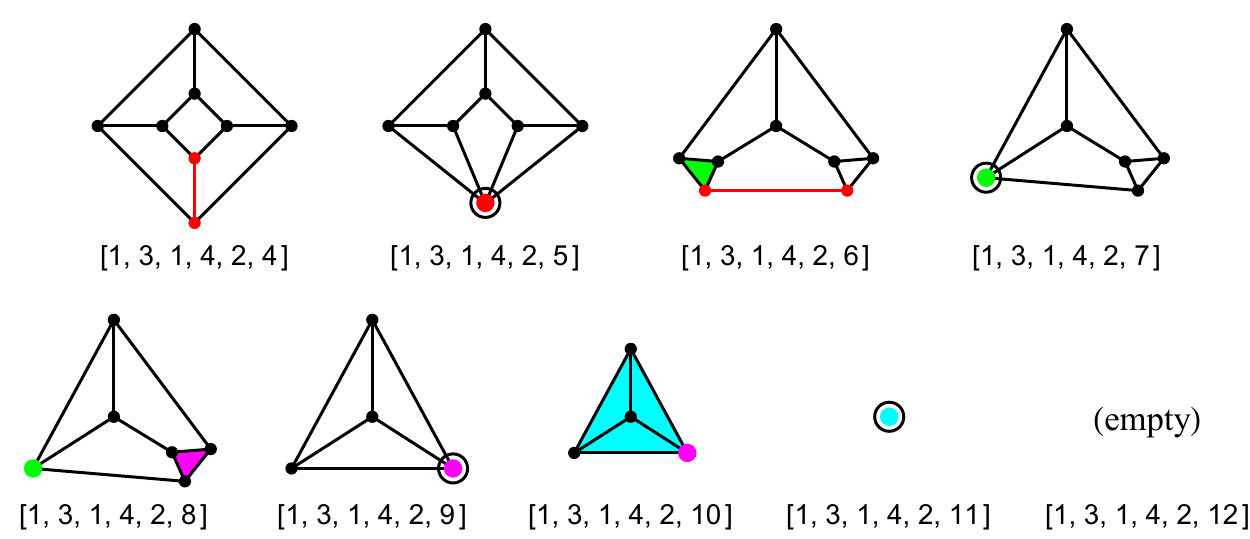}
	\end{center}
	\vspace*{-5 pt}
	\caption{The combinatorial structures of the fixed-angles polyhedra $\ConvEuc^6[1,3,1,4,2,\sigma]$ for integral $\sigma$ from~$4$ to $12$.  The four digon vertices are circled.  At $\sigma=5$, a red face of dimension $1$ becomes a new face of dimension $1$; at $\sigma=7$ and at $\sigma=9$, green and magenta faces of dimension $2$ become dimension $0$; and at $\sigma=11$, the cyan face of dimension $3$ becomes dimension $-1$.}
	\label{fig:Polyhedra}
\end{figure}

For some examples of this, Figure~\ref{fig:Polyhedra} depicts how the combinatorial structure of the polyhedron $\ConvEuc^6[1,3,1,4,2,\sigma]$ changes as $\sigma$ increases from $4$ to $12$, mostly simplifying as the necklace becomes more imbalanced:
\begin{enumerate}
\item[4:]  The initial $\ConvEuc^6[1,3,1,4,2,4]$ has the combinatorial structure of a cube.  
\item[5:] When $\sigma=5$, the substring tie of width $w=3$ between $(3,1,4)$ and $(2,5,1)$ produces a digon vertex where four face-planes concur.  As $\sigma$ passes through~$5$, the red edge of the cube shrinks to that digon vertex and reemerges as a new edge, but connected differently.  A $1$\_face reemerges as a $1$\_face because $n-w-2$ and $w-2$ are both $1$; so the combinatorial structures before and after, while different, are numerically equally complicated.
\item[6:] The polyhedron $\ConvEuc^6[1,3,1,4,2,6]$ can be thought of as a tetrahedron with two of its four vertices truncated.  This polyhedron is also $C_3(6)^\Delta$, the dual of the convex hull of six points along a twisted cubic.
\item[7:] When $\sigma=7$, the substring tie of width $2$ between $(1,3,1,4)$ and $(2,7)$ produces a digon vertex that is simple.  As $\sigma$ passes through~$7$, the green triangular face shrinks to that vertex and then reemerges as a trigon vertex; so a $2$\_face becomes a $0$\_face.
\item[8:] The polyhedron $\ConvEuc^6[1,3,1,4,2,8]$ is a tetrahedron with one vertex truncated.
\item[9:] When $\sigma=9$, the substring tie of width $2$ between $(3,1,4,2)$ and $(9,1)$ produces another simple digon vertex.  As $\sigma$ passes through~$9$, the magenta triangular face of $C_3(6)^\Delta$ shrinks to that digon vertex and emerges as a trigon vertex.
\item[10:] The polyhedron $\ConvEuc^6[1,3,1,4,2,10]$ is a tetrahedron.
\item[11:] When $\sigma=11$, the substring tie of width $1$ between $(1,3,1,4,2)$ and $(11)$ produces a digon vertex that is an isolated point.  As $\sigma$ passes through~$11$, the entire tetrahedron shrinks to that vertex and disappears, a $3$\_face becoming a $(-1)$\_face.
\item[12:] The polyhedron $\ConvEuc^6[1,3,1,4,2,12]$ is empty.
\end{enumerate}

\section{Dual cyclic when {\large $\protect\nlow$} is odd}
\label{sect:nOdd}

The fixed-angles polytopes $\ConvEuc^n[\mathbf{s}]$ are combinatorially richer when the entries of the necklace $\mathbf{s}$ are well balanced.  When $n$ is odd, the richest of all are the duals of cyclic polytopes.

\subsection{Majority dominance} 
Given a shangle necklace $\mathbf{s}$ of length $n$, we call a substring of~$\mathbf{s}$ \emph{short}, \emph{diametral}, or \emph{long} according as its length is less than, equal to, or greater than $n/2$.  We call a shangle necklace \emph{majority dominant} when all of its short substrings are light or, equivalently, when all of its long substrings are heavy.  Majority dominance doesn\t say anything about the weights of the diametral substrings, which arise only when $n$ is even.  Any necklace of the form $[1\pm\epsilon,1,\ldots,1]$ for $0\le\epsilon<1$ is majority dominant; but the entries of a majority-dominant necklace can also vary widely, as in the length\_$7$ necklace $[1,1,M,1,1,1,M]$ for large $M$.  

For $n$ of either parity, majority dominance has a consequence worth noting:

\begin{lemma}
\label{lem:Neighborly}
When the necklace $\mathbf{s}$ is majority dominant, any set of at most $(n-3)/2$ of the facets of the fixed-angles polytope $\ConvEuc^n[\mathbf{s}]$ always intersect in a nonempty face.  Stating that more concisely, the $(n-3)$\_dimensional polytope $\ConvEuc^n[\mathbf{s}]^\Delta$ is neighborly.
\end{lemma}

\begin{proof}
A facet of $\ConvEuc^n[\mathbf{s}]$ is the subset where some edge has shrunk to a point.  The only thing that can prevent some set of edges from shrinking to points simultaneously is when that set includes a string of consecutive edges whose simultaneous shrinking would produce a supervertex of shangle more than $S/2$.  When the necklace is majority dominant, at least $n/2$ vertices would have to merge to produce such a supervertex, and that would require the shrinking of at least $(n-2)/2$ edges.  But we are discussing here what happens when the number that shrink is at most $(n-3)/2$.
\end{proof}

We turn next to the consequences of majority dominance when $n$ is odd.  

\subsection{Tours, small steps, and odd steps}  

Let $n\ge 3$ be odd.  We define a \emph{tour} to be a length-$3$ necklace of distinct elements of $\ints/n\ints$.  Given some tour $[p,q,r]$, suppose that we travel around $\ints/n\ints$ in the increasing direction by stepping from $p$ to $q$, on to $r$, and then around to $p$ again.  We classify a tour as \emph{once-around} or \emph{twice-around}, according as those steps take us once or twice around $\ints/n\ints$.  In a twice-around tour, each step from one element to another passes over the third on the way. 

There are lots of ways to write down any tour as an ordered triple of integers $(i,j,k)$.  We can replace $i$ by $i+kn$ for any $k$ without changing the corresponding element of $\ints/n\ints$; and we can choose any one of the three cyclic shifts $(i,j,k)$, $(j,k,i)$, or $(k,i,j)$.  Among all of those alternatives, each tour has a unique \emph{normal form} $(i,j,k)$, in which $i$, $j$, and $k$ all lie in $[1\dd n]$ and in which $i$ is the smallest of the three.  The normal form $(i,j,k)$ of a once-around tour has $1\le i<j<k\le n$, while that of a twice-around tour has $1\le i<k<j\le n$.

Let $m\coloneq (n-1)/2$, so that $n=2m+1$ with $m\ge 1$.  We define four sets of integers, each of cardinality $m$:
\begin{align*}
S&\coloneq \{1,2,\ldots,m\}\\
L&\coloneq \{m+1,m+2,\ldots,2m\},\\
O&\coloneq \{1,3,5,\ldots,n-2\},\text{ and}\\
E&\coloneq \{2,4,6,\ldots,n-1\}.
\end{align*}
Any nonzero element of $\ints/n\ints$ is congruent, modulo $n$, either to an element of~$S$ or to an element of~$L$, so we call it either \emph{small} or \emph{large}.  It is also congruent either to an element of~$O$ or of~$E$, so we call it \emph{odd} or \emph{even}.  

It is easy to see that $-S=L$, $-O=E$, $2S=E$, and $2L=O$.  Since the multiplicative inverse of $2$ in the ring $\ints/n\ints$ is $-m=(1-n)/2$, we also have $(-m)E=S$ and $(-m)O=L$, which implies $mE=L$ and $mO=S$.

The three \emph{steps} of a tour $[p,q,r]$ are the three nonzero differences $q-p$, $r-q$, and $p-r$.  We call a tour \emph{small} when its three steps are all small elements of $\ints/n\ints$; and we call it \emph{odd} when its three steps are all odd.  Since three elements of $S$ always sum to less than $2n$, every small tour goes around $\ints/n\ints$ once.  Three elements of $O$ can\t sum to $2n$ either, since their sum is odd; so every odd tour also goes around $\ints/n\ints$ once.

\begin{lemma}
\label{lem:TwiceSmallIsOdd}
The map $[p,q,r]\mapsto [-2p,-2q,-2r]$ is a bijection from the set of small tours to the set of odd tours.  Its inverse is the map $[p,q,r]\mapsto [mp,mq,mr]$.
\end{lemma}

\begin{proof}
Multiplication by $-2$ takes $S$ to $O$, while its inverse, which is multiplication by~$m$, takes $O$ back to $S$.  And both of those operations distribute over the subtraction that computes a difference.
\end{proof}

It follows that there just as many small tours as odd tours.  You might enjoy verifying that there are $\frac{1}{4}\binom{n+1}{3}$ of each.  That expression also counts, in a regular $n$\_gon, the number of inscribed triangles that contain the $n$\_gon\s center, since the two orientations of any inscribed triangle give us two tours, and that triangle contains the center just when one of those tours is small.

\subsection{Showing that two face lattices are dual}

\begin{lemma}
\label{lem:FixedAngles}
When $n\ge 3$ is odd and the shangle necklace $\mathbf{s}$ is majority dominant, three edges $\edge_i$, $\edge_j$, and $\edge_k$ of an \shatgon with $1\le i<j<k\le n$ can be extended to form a trigon just when the triple $(i,j,k)$ is the normal form of a small tour.
\end{lemma}

\begin{proof}
The three substrings into which the shangle necklace $\mathbf{s}$ is cut by the edges $\edge_i$, $\edge_j$, and $\edge_k$ are 
\[(s_i,\ldots,s_{j-1}),\quad (s_j,\ldots,s_{k-1}), \quad\text{and}\quad (s_k,\ldots,s_n,s_1,\ldots,s_{i-1}),\]
 of lengths $j-i$, $k-j$, and $(n+i)-k$.  Those edges can be extended to form a trigon just when those three substrings are all light.  When the necklace $\mathbf{s}$ is majority dominant, that happens just when those substrings are all short, that is, when their lengths are less than $n/2$.  And that happens just when the triple $(i,j,k)$ is the normal form of a small tour.
\end{proof}

\begin{lemma}
\label{lem:Cyclic}
When $n\ge 5$ is odd, three vertices $v_i$, $v_j$, and $v_k$ of the cyclic polytope $C_{n-3}(n)$ with $1\le i<j<k\le n$ lie on the same side of the hyperplane determined by the other $n-3$ vertices just when the triple $(i,j,k)$ is the normal form of an odd tour.
\end{lemma}

\begin{proof}
This is essentially Gale\s evenness condition~\cite[pp.~14--15]{Ziegler}.  

The moment curve passes through the vertices $v_1$ through $v_n$ of the cyclic polytope $C_{n-3}(n)$ in that order.  The $n-3$ vertices other than $v_i$, $v_j$, and $v_k$ determine a hyperplane, which the moment curve intersects at those $n-3$ vertices.  The vertices $v_i$, $v_j$, and $v_k$ lie on the same side of that hyperplane just when the moment curve intersects it an even number of times during each passage from one of those three to another.

In passing from $v_i$ to $v_j$, the moment curve intersects the hyperplane $j-i-1$ times, at $v_{i+1}$ through $v_{j-1}$.  In passing from $v_j$ to $v_k$, that number is $k-j-1$.  By Gale\s evenness condition, the vertices $v_i$, $v_j$, and $v_k$ lie on the same side of that hyperplane just when $j-i-1$ and $k-j-1$ are even, that is, when $j-i$ and $k-j$ are odd.  But whenever those two are both odd, $(n+i)-k$ must be odd as well.  So $v_i$, $v_j$, and $v_k$ lie on the same side of the hyperplane just when the triple $(i,j,k)$ is the normal form of an odd tour.
\end{proof}

\begin{prop}
\label{prop:EquiOdd}
When $n\ge 5$ is odd and the necklace $\mathbf{s}$ is majority dominant, the face lattice of the fixed-angles polytope $\ConvEuc^n[\mathbf{s}]$ is dual to that of the cyclic polytope $C_{n-3}(n)$.
\end{prop}

\begin{proof}
It suffices to establish a one-to-one correspondence between the $n$ facets of $\ConvEuc^n[\mathbf{s}]$ and the $n$ vertices of $C_{n-3}(n)$ with the property that $n-3$ facets of $\ConvEuc^n[\mathbf{s}]$ intersect at a vertex precisely when the corresponding $n-3$ vertices of $C_{n-3}(n)$ lie on a common facet.  We establish such a correspondence by associating the facet of $\ConvEuc^n[\mathbf{s}]$ on which the edge $\edge_i$ shrinks to have length $0$ with the vertex $v_{i'}$ of the cyclic polytope $C_{n-3}(n)$ just when $i'\equiv -2i \pmod{n}$. 

By Lemma~\ref{lem:FixedAngles}, the $n-3$ facets of $\ConvEuc^n[\mathbf{s}]$ other than those associated with the edges $\edge_i$, $\edge_j$, and $\edge_k$ intersect at a trigon vertex just when the triple $(i,j,k)$ is the normal form of a small tour.  By Lemma~\ref{lem:Cyclic}, the $n-3$ vertices of $C_{n-3}(n)$ other than $v_{i'}$, $v_{j'}$, and $v_{k'}$ lie on a common facet just when the triple $(i',j',k')$ is the normal form of an odd tour.  Finally, by Lemma~\ref{lem:TwiceSmallIsOdd}, the tour $[i,j,k]$ is small just when the tour $[-2i,-2j,-2k]$ is odd.
\end{proof}

\section{Dual cyclic when {\large $\protect\nlow$} is even}
\label{sect:nEven}

Suppose now that $n$ is even.  We continue to require that the necklace $\mathbf{s}$ be majority dominant, so the polytope $\ConvEuc^n[\mathbf{s}]^\Delta$ is neighborly by Lemma~\ref{lem:Neighborly}.  Majority dominance forbids substring ties of width $w<n/2$, but allows ties of width $w=n/2$.  If we further require that~$\mathbf{s}$ be entirely free of substring ties, then the polytope $\ConvEuc^n[\mathbf{s}]$ will be simple, so its dual $\ConvEuc^n[\mathbf{s}]^\Delta$ will be simplicial, as well as neighborly --- and hence max-faced, by the upper-bound theorem~\cite[p.~254]{Ziegler}.

When $n$ is even, what further condition can we impose on $\mathbf{s}$ to guarantee that $\ConvEuc^n[\mathbf{s}]^\Delta$, beyond being max-faced, is actually the cyclic polytope $C_{n-3}(n)$?

\subsection{Dipole tie-breaking} 
As long as none of the $n$ diametral substrings are tied, which $n/2$ of them are light and which heavy doesn\t affect how many faces of each dimension the fixed-angles polytope has, but it does affect how those faces connect up.  In the necklace $[1,3,1,4,2,4]$ that starts Figure~\ref{fig:Polyhedra}, for example, the heavy diametral substrings are $(3,1,4)$, $(4,2,4)$, and $(4,1,3)$, distributed symmetrically around the necklace; and the resulting polytope $\ConvEuc^6[1,3,1,4,2,4]$ is combinatorially a cube.  In the later necklace $[1,3,1,4,2,6]$, however, the heavy triples are the three that contain~$6$, distributed quite asymmetrically; and the resulting polytope $\ConvEuc^6[1,3,1,4,2,6]$ is $C_3(6)^\Delta$, the dual of a cyclic polytope.  That maximal asymmetry turns out to be the key.

We say that $\mathbf{s}$ has \emph{dipole tie-breaking} when it has a pair of antipodal entries, its \emph{light pole} and \emph{heavy pole}, with the property that the $n/2$ diametral substrings that contain the light pole are all light and the $n/2$ that contain the heavy pole are all heavy.  For example, the necklace $[1,2,1,4,3,4]$ has dipole tie-breaking with light pole $2$ and heavy pole $3$.  This example shows that the light pole need not be a least entry and the heavy pole need not be a greatest.  The heavy pole must exceed the light pole, however.

Dipole tie-breaking is the extra requirement that we impose on the necklace~$\mathbf{s}$ when~$n$ is even so that the fixed-angles polytope $\ConvEuc^n[\mathbf{s}]$ will be dual cyclic.  To show this, we begin with a lemma about the vertex figures of cyclic polytopes.

\begin{lemma}
\label{lem:VertexFigures}
For any $d$ and $m$, the vertex figure of the last vertex of the cyclic polytope $C_d(m)$ is itself cyclic, being an instance of $C_{d-1}(m-1)$.  If $d$ is even, the vertex figures of all of the vertices of $C_d(m)$ are cyclic.
\end{lemma}

\begin{proof}
We can represent a subset of the vertices of a cyclic polytope as a string of bits, with $1$\s for vertices in the subset and $0\mkern 1mu$\s for the ones not in it, in order along the moment curve.  By Gale\s evenness condition, a set of vertices of the right cardinality to be a facet actually forms a facet just when, for every two $0\mkern 1mu$\s in that bit string, the number of 1\s between them is even.  But adding a new $1$ at the end of the bit string can\t affect that test.  Thus, if $V$ is some subset of the vertices of $C_{d-1}(m-1)$ with $\lvert V\rvert=d-1$, then $V$ is a facet of $C_{d-1}(m-1)$ just when $V\cup\{v_m\}$ is a facet of $C_d(m)$ containing~$v_m$.  So the vertex figure of~$v_m$ in $C_d(m)$ is a $C_{d-1}(m-1)$.  

When $d$ is even, there are combinatorial automorphisms of $C_d(m)$ that permute its $m$ vertices dihedrally; so all vertices look the same.
\end{proof}

\begin{prop}
\label{prop:EquiEven}
If a necklace $\mathbf{s}$ has even length $n\ge 4$, is majority dominant, and has dipole tie-breaking, then the dual $\ConvEuc^n[\mathbf{s}]^\Delta$ of its fixed-angles polytope is an instance of the cyclic polytope $C_{n-3}(n)$.
\end{prop}

\begin{proof}
It suffices to show this for the particular necklace $\mathbf{s}=[1,\ldots,1,2]$, which is majority dominant and has dipole tie-breaking, since the combinatorial structure of $\ConvEuc^n[\mathbf{s}]$ depends only upon which substrings of $\mathbf{s}$ are light, tied, and heavy.  Going up by one dimension, the facial lattice of $\ConvEuc^{n+1}[1,\ldots,1]$ is dual to that of the cyclic polytope $C_{n-2}(n+1)$, by Prop.~\ref{prop:EquiOdd}.  Under that duality, each of the facets of $\ConvEuc^{n+1}[1,\ldots,1]$ is dual to one of the vertex figures of $C_{n-2}(n+1)$, which is a $C_{n-3}(n)$ by Lemma~\ref{lem:VertexFigures}.  But such a facet corresponds to the $(n+1)$\_gons in which a particular edge has shrunk to length~$0$.  That causes two vertices to merge into a supervertex of shangle $2$, so each facet of $\ConvEuc^{n+1}[1,\ldots,1]$ is a copy of $\ConvEuc^{n}[1,\ldots,1,2]$.
\end{proof}

\subsection{Achieving dihedral symmetry}

When $n$ is even, it is intriguing to consider how the fixed-angles polytope $P_\epsilon\coloneq \ConvEuc^n[1,\ldots,1,1+\epsilon]$ varies with $\epsilon$.  For $0<\epsilon<2$, that shangle necklace is majority dominant and has dipole tie-breaking, so the polytope $P_\epsilon$ is a $C_{n-3}(n)^\Delta$; and, since $n-3$ is odd, it has only four automorphisms.  When $\epsilon$ drops to $0$, however, the polytope $P_\epsilon$ must acquire $2n$ symmetries that permute its $n$ facets dihedrally.  Those symmetries appear because the $n/2$ substring ties that arise when~$\epsilon$ reaches~$0$ cause $n/2$ faces of $P_\epsilon$ to shrink from small instances of $\Delta_{n/2-2}$ to points.  For example, when $n=6$, the polyhedron $C_3(6)^\Delta$ has eight vertices and six faces: two triangles, two quadrilaterals, and two pentagons.  When $\epsilon$ reaches $0$, three of its edges shrink to points, those points becoming the vertices of the triangular base of the bipyramid $\ConvEuc^6[1,\ldots,1]$.\footnote{Of the three edges of $C_3(6)^\Delta$ that shrink, no two are adjacent, two bound each of the pentagonal faces, and one bounds each quadrilateral face.  There are two ways to choose such a triple of edges.}

\section{Fixing the edge lengths}
\label{sect:FixedLengths}
People fix the edge lengths of their polygons more often than they fix the vertex angles.  What can we say about that situation?

\subsection{The Kapovich--Millson homeomorphism}

Given some necklace $\mathbf{s}=[s_1,\ldots,s_n]$ of positive numbers with $S\coloneq s_1+\cdots+s_n$, we define an \emph{\sbargon} to be a planar $n$\_gon whose $k^\text{th}$ edge $\edge_k$ has length $\lvert\edge_k\rvert=s_k$, for all $k$.  The bar accent indicates that it is the edge lengths that we here constrain.  The polygon space $\PolygonSpace^n[\mathbf{s}]$ is the moduli space of shapes of \sbargon{}s, while $\PolygonSpace_c^n[\mathbf{s}]$ is the subset in which the \sbargon{}s are ccw-convex.
 
Given an \sbargon that is ccw-convex, Kapovich and Millson~\cite{KapMill95} associate the number~$2\pi s_k/S$ with the point on the unit circle where the normalized vector $\edge_k/s_k$ ends.  They then use a Schwarz--Christoffel map to conformally transform the unit disk into an $n$\_gon, the $k^\text{th}$ of those points becoming a vertex with external angle $2\pi s_k/S$.  That $n$\_gon is defined only up to an orientation-preserving similarity; but we can scale it to have unit perimeter (or unit area).  We thus transform each ccw-convex \sbargon into an \shatgon, getting a map $h_\mathbf{s}\colon \PolygonSpace_c^n[\mathbf{s}]\to\Conv^\nthin[\mathbf{s}]$ that Kapovich and Millson show to be a homeomorphism.  The convex-polygon space $\PolygonSpace_c^n[\mathbf{s}]$ is thus homeomorphic to a polytope by a preferred homeomorphism, which makes it a \emph{topological polytope}.  And we conclude:

\begin{cor}
When $n$ is odd and the fixed lengths in $\mathbf{s}$ are majority dominant, the subset $\PolygonSpace_c^n[\mathbf{s}]$ of the polygon space $\PolygonSpace^n[\mathbf{s}]$ in which the $n$\_gons are ccw-convex is a topological polytope that is combinatorially $C_{n-3}(n)^\Delta$.  The same holds when~$n$ is even if the fixed lengths in $\mathbf{s}$ also have dipole tie-breaking.
\end{cor}

\subsection{Ties, singular points, and corners}

Given some $n$\_gon, we can produce various $n$\_gons by reassembling its $n$ directed edges, tip to tail, in any of $(n-1)!$ cyclic orders.  Some assembly order always produces an $n$\_gon that is ccw-convex; so the polygon space $\PolygonSpace^n[\mathbf{s}]$ is covered by the fixed-lengths-convex top-polytopes $\PolygonSpace^n_c[\mathbf{s}']$, for the $(n-1)!$ cyclic reorderings~$\mathbf{s}'$ of~$\mathbf{s}$.  Those top-polytopes overlap just on their faces, in each of which some sets of edges share a common phase (and hence function as \emph{superedges}).

A \emph{subset tie} arises in the length vector $\mathbf{s}$ when some $w\le n/2$ of the lengths, not necessarily consecutive, have the same sum as the remaining $n-w$.  Each subset tie in~$\mathbf{s}$ generates a \emph{quadratic singular point} of the polygon space $\PolygonSpace^n[\mathbf{s}]$.  Such a singular point has a neighborhood in $\PolygonSpace^n[\mathbf{s}]$ that is analytically isomorphic to a neighborhood of the origin in the zero set of a quadratic form on $\mathbb{R}^{n-2}$ whose signature is $(w-1,n-w-1)$.\footnote{Kapovich and Millson only sketch the structure of the singular points in~\cite{KapMill95}, but the methods that they use in~\cite{KapMill97,KapMill99} to show that the singularities in a moduli space of spherical linkages are Morse, and hence quadratic, apply to any space of constant curvature.}  

When the length vector~$\mathbf{s}$ has such a subset tie, there are $w!(n-w)!$ cyclic reorderings~$\mathbf{s}'$ of~$\mathbf{s}$ that convert both of the tied subsets into substrings, thus giving a necklace~$\mathbf{s}'$ with a substring tie of width $w$.  
That substring tie causes the fixed-angles polytope $\ConvEuc^n[\mathbf{s}']$ to have a digon vertex, and we refer to the corresponding point of the fixed-lengths-convex top-polytope $\PolygonSpace^n_c[\mathbf{s}']$ as a \emph{digon corner}.  The quadratic singular point of a polygon space $\PolygonSpace^n[\mathbf{s}]$ that arises from a subset tie of width $w$ in $\mathbf{s}$ is thus a digon corner of each of the $w!(n-w)!$ topological polytopes $\PolygonSpace^n_c[\mathbf{s}']$ that touch it.

\subsection{Diffeomorphism issues}

When the necklace $\mathbf{s}$ is free of substring ties, the fixed-angles polytope $\ConvEuc^n[\mathbf{s}]$ has only simple vertices, so it is a smooth $(n-3)$\_manifold with corners that happens to be flat.  As for the fixed-lengths-convex top-polytope $\PolygonSpace_c^n[\mathbf{s}]$, near any point in it, the $n$\_gon has at least  three external angles that are positive, and we can use the other $n-3$ external angles as local coordinates, constrained to be nonnegative.  So the top-polytope $\PolygonSpace_c^n[\mathbf{s}]$ is also a smooth $(n-3)$\_manifold with corners.  Those two spaces seem likely to be diffeomorphic, and perhaps some diffeomorphism can be shown to exist using obstruction theory~\cite{Davis}. 

When the necklace $\mathbf{s}$ has substring ties, there is no hope for a diffeomorphism.  That is easy to see for a substring tie of width $w\ge 3$.  The resulting digon vertex of the polytope $\ConvEuc^n[\mathbf{s}]$ is nonsimple; so both that polytope and the homeomorphic top-polytope $\PolygonSpace_c^n[\mathbf{s}]$ are not smooth manifolds with corners.\footnote{Joyce~\cite{Joyce} proposes a more general notion of a \emph{smooth manifold with g-corners}; but the top-polytope $\PolygonSpace_c^6[1,\ldots,1]$ for ccw-convex equilateral hexagons isn\t one of those either.}  For a substring tie of width $w=2$, the digon vertex of $\ConvEuc^n[\mathbf{s}]$ is simple, so that polytope might still be a smooth manifold with corners.  Once $n\ge 5$, however, the top-polytope $\PolygonSpace_c^n[\mathbf{s}]$ is bent, near its digon corner, in a way that prevents it from being a smooth manifold with corners; so the situation remains hopeless.

We close with a warning:  Even when $\mathbf{s}$ is free subset ties, so that some diffeomorphism likely exists, the Kapovich--Millson map $h_\mathbf{s}\colon \PolygonSpace_c^n[\mathbf{s}]\to\ConvEuc^n[\mathbf{s}]$ is not a diffeomorphism.  Consider an \sbargon that approaches the boundary of the top-polytope $\PolygonSpace_c^n[\mathbf{s}]$, say because its external angle $\rho$ at the vertex joining its edges of lengths $s_k$ and $s_{k+1}$ approaches zero.  Let $r$ be the length of the edge that joins the vertices with angles $s_k$ and $s_{k+1}$ in the corresponding \shatgon.  As the angle $\rho$ goes to zero, the length $r$ does also, but at a slower rate.  By the Schwarz--Christoffel integral, $r$ goes to zero like $\rho^{\mkern 1mu q}$ where $q=1-2(s_k+s_{k+1})/S$, and $q<1$ prevents the map~$h_\mathbf{s}$ from being a diffeomorphism.


\end{document}